\documentclass[11pt]{article}
\usepackage[english]{babel}
\usepackage{amsmath,amsthm,amsfonts,amssymb,epsfig,verbatim}
\usepackage[left=1in,top=1in,right=1in]{geometry}
\usepackage[normalem]{ulem}

\newtheorem{thm}{Theorem}
\newtheorem{lem}{Lemma}

\newtheorem{rem}{Remark}

\newcommand{\e}{\mathbb{E}}

\newcommand{\la}{\langle}
\newcommand{\ra}{\rangle}

\DeclareMathOperator{\Var}{Var}

\begin{document}
	
    \title{Order of Fluctuations of the Free Energy in \\
    	the SK Model at Critical Temperature}
	\author{Wei-Kuo Chen \thanks{School of Mathematics, University of Minnesota. Email: wkchen@umn.edu} \and Wai-Kit Lam \thanks{School of Mathematics, University of Minnesota. Email: wlam@umn.edu}}
	
	\date{}
	
	\maketitle
	
	\begin{abstract}	
    We present an elementary approach to the order of fluctuations for the free energy in the Sherrington-Kirkpatrick mean field spin glass model at and near the critical temperature. It is proved that at the critical temperature the variance of the free energy is of $O((\log N)^2).$ In addition, we show that if one approaches the critical temperature from the low temperature regime at the rate $O(N^{-\alpha})$ for some $\alpha>0,$ then the variance is of  $O((\log N)^2+N^{1-\alpha}).$
	\end{abstract}

	
	\section{Introduction}
	
	The Sherrington-Kirkpatrick (SK) model was initially introduced in 1975 \cite{SK} in order to explain some strange magnetic properties of certain alloys. Over the past decades, it has received a great attention in the physics and mathematics communities. See \cite{MPV} for physics treatments and \cite{Pan,Talagrand,TalagrandBook1,TalagrandBook} for recent mathematical progress. 
	
	The aim of this short note is to study the order of fluctuations of the free energy in the SK model at the critical temperature. For any $N\geq 1,$ the Hamiltonian of the SK model is defined as 
	\[
	H_N(\sigma) = \frac{1}{\sqrt{N}} \sum_{1\leq i, j\leq N} g_{ij}\sigma_i\sigma_j
	\]
	for any $\sigma\in \Sigma_N:=\{-1,+1\}^N,$
	where $(g_{ij})_{1\leq i, j\leq N}$ is a family of independent standard Gaussian random variables. The covariance of $H_N$ is described by 
	$$
	\e H_N(\sigma^1)H_N(\sigma^2)=NR(\sigma^1,\sigma^2)^2,
	$$
	where $R(\sigma^1,\sigma^2):=N^{-1}\sum_{i=1}^N\sigma_i^1\sigma_i^2$ is called the overlap between spin configurations $\sigma^1,\sigma^2\in \Sigma_N.$ 
	The free energy and the Gibbs measure at (inverse) temperature $\beta>0$ is defined as
	\begin{align*}
	F_N(\beta) = \log \sum_{\sigma\in \Sigma_N} \exp\bigl(\beta H_N(\sigma)\bigr)
	\end{align*}
	and
	\begin{align*}
	G_N(\sigma)=\frac{\exp\bigl(\beta H_N(\sigma)\bigr)}{Z_N(\beta)},\,\,\forall \sigma\in \Sigma_N,
	\end{align*}
	where $Z_N(\beta):=\sum_{\sigma\in \Sigma_N} \exp\bigl(\beta H_N(\sigma)\bigr).$
	Denote by $\sigma^1,\sigma^2$ i.i.d. samplings from $G_N$ and by $\la \cdot\ra$ the expectation with respect to these random variables. 
	
	The SK model is known to exhibited a phase transition at the critical temperature $\beta_c:=1/\sqrt{2}$. In the high temperature regime $\beta<\beta_c,$ the limiting free energy is equal to the anneal free energy, that is, $$
	\lim_{N\to\infty}\frac{F_N(\beta)}{N}=\lim_{N\to\infty}\frac{\log \e Z_N }{N}
	$$
	and if we sample two independent $\sigma^1,\sigma^2$ from the Gibbs measure, then they are essentially orthogonal to each other in the sense that $\lim_{N\to\infty}\e\la R(\sigma^1,\sigma^2)^2\ra=0.$ In contrast, the model exhibits different behaviors in the lower temperature regime $\beta>\beta_c,$ where we see that
	$$
	\lim_{N\to\infty}\frac{F_N(\beta)}{N}<\lim_{N\to\infty}\frac{\log \e Z_N }{N}
	$$
	and the two independent samplings $\sigma^1,\sigma^2$ have nonzero overlap, that is, $\lim_{N\to\infty}\e\la R(\sigma^1,\sigma^2)^2\ra>0.$ Indeed, it was conjectured that the limiting distribution of $R(\sigma^1,\sigma^2)$ should be be described by a probability measure supported on an interval $[0,q]$ for some $q\in (0,1).$ See \cite{MPV,Talagrand,TalagrandBook} for more details.
	
	The order of fluctuations of the free energy is summarized as follows:
	
	\begin{itemize}
	
	\item {\bf High temperature}: Aizenman, Lebowitz, and Ruelle \cite{AizemanHighTemp} proved that the free energy has Gaussian fluctuations in the high temperature regime. Their result implies that the limit of $\mbox{Var}(F_N(\beta))$ exists and is finite. 
	
    \item {\bf Near critical temperature}: The problem of understanding the transition near the critical temperature was intensively studied in Talagrand's books \cite[Section 2.14]{Talagrand} and \cite[Section 11.7]{TalagrandBook}, where he showed that when $\beta=\beta_N$ approaches the critical temperature from the high temperature regime in the rate $\lim_{N\to\infty}N^{1/3}(\beta_c^2-\beta_N^2)=c$, then the overlap undergoes a phase transition depending on whether $c$ is finite or infinite. Exactly at the criticality $\beta_c,$ he also proved that the overlap is controlled by $\e \la R(\sigma^1,\sigma^2)^2\ra\leq C/\sqrt{N}$ for some universal constant $C$, from which it can be derived that $\mbox{Var}(F_N(\beta_c))\leq C\sqrt{N}$, see Remark \ref{rmk1} below. 
    
    \item {\bf Low temperature}: Chatterjee \cite{Chatterjee1} showed that $\mbox{Var}(F_N(\beta))\leq C(\beta)N/\log N$, where $C(\beta)$ is a constant independent of $N.$ Incidentally, it was proved by Chatterjee \cite[Theorem~2.4]{ChatterjeeFluctuations} that fluctuations of the free energy are at least of order $1$ at any $\beta$. 
     
     \end{itemize}
 
     Even though the SK model was proposed more than 40 years ago, understanding the exact order of fluctuations of its free energy at the critical temperature as well as in the lower temperature regime remains an open and very challenging question in spin glasses.
      
     This note is focused on the order of fluctuation of $F_N$ at the critical temperature. This case was studied before in the physics literature \cite{ASP,PR}, where it was expected that
     $$
     \mbox{Var}(F_N(\beta_c))=\frac{1}{6}\log N+O(1).
     $$
  	 Our main result below contains two parts. First, at the critical temperature, we obtain an upper bound of order $(\log N)^2$ for the variance of the free energy. Second, we show that if one approaches the critical temperature from the low temperature regime in the rate $\beta_N=\sqrt{\beta_c^2+N^{-\alpha}}$ for $\alpha\in (0,1)$, then a polynomial bound $N^{1-\alpha}$ is obtained. This improves Chatterjee's bound $N/\log N$ in the near critical case.

	\begin{thm}
		\label{thm: sk}
		The following statements hold:
		\begin{enumerate}
			\item There exists a constant $C>0$ such that
			$$
			\Var\bigl(F_N\bigl(\beta_c\bigr)\bigr)\leq C\bigl((\log N)^2+1\bigr),\,\,\forall N\geq 1.
			$$
			\item For any fixed $\alpha>0$ and $d>0,$ there exists a constant $C$ depending only on $\alpha$ and $d$ such that
			$$
			\Var\left(F_N\left(\sqrt{\beta_c^2+dN^{-\alpha}}\right)\right)\leq C\bigl((\log N)^2+N^{1-\alpha}\bigr),\,\,\forall N\geq 1.
			$$
		\end{enumerate}
	\end{thm}
         
    Our approach is motivated by a work of Guerra and Toninelli \cite{GuerraToninelli}, where they derived the limit of the SK free energy and provided a rate of convergence. Their idea was to consider a coupled free energy with Hamiltonian of the form $\sqrt{s}(H_N(\sigma^1)+H_N(\sigma^2))+\lambda NR(\sigma^1,\sigma^2)^2$, where $\sqrt{s}$ is a varying temperature and $\lambda$ is an auxiliary parameter. From this, they derived an ordinary differential inequality for this coupled free energy in the variable $s\geq 0$ and by solving this inequality, they obtained the rate of convergence of the free energy. Our argument adopts a different route by considering a coupled Hamiltonian that is related to the problem of chaos in disorder in the SK model considered in \cite{Chatterjee1,ChatterjeeSuperConc}.

	\section{Proof of Theorem~\ref{thm: sk}}
	
		Let $H_N',H_N''$ be two independent copies of $H_N.$ For $0\leq t\leq 1$ and $\sigma, \rho\in \Sigma_N$, set
	\begin{align*}
	H_{N,t}^1(\sigma)&=\sqrt{t}H_N(\sigma)+\sqrt{1-t}H_N'(\sigma),\\
	H_{N,t}^2(\rho)&=\sqrt{t}H_N(\rho)+\sqrt{1-t}H_N''(\rho).
	\end{align*}
	Note that $$\e H_{N,t}^1(\sigma)H_{N,t}^2(\rho)=tNR(\sigma,\rho)^2.$$
	For any $t\in [0,1]$ and $\lambda\in \mathbb{R}$, define
	\begin{align*}
	\phi_N(t,\lambda)&=\frac{1}{N}\e\log \sum_{\sigma,\rho\in \Sigma_N}\exp\bigl(\beta(H_{N,t}^1(\sigma)+H_{N,t}^2(\rho))+\lambda \beta^2NR(\sigma,\rho)^2\bigr).
	\end{align*}
	Denote by $\la \cdot\ra_{t,\lambda}$ the Gibbs expectation associated to this free energy. That is, the expectation with respect to the measure
	\[
	\frac{\sum_{\sigma,\rho}I((\sigma, \rho)\in \cdot) \exp\bigl(\beta(H_{N,t}^1(\sigma)+H_{N,t}^2(\rho))+\lambda \beta^2NR(\sigma,\rho)^2\bigr)}{\sum_{\sigma,\rho}\exp\bigl(\beta(H_{N,t}^1(\sigma)+H_{N,t}^2(\rho))+\lambda \beta^2NR(\sigma,\rho)^2\bigr)}.
	\]
	
	\begin{lem}
		For any $\beta>0$ and $t\in [0,1]$  satisfying $2\beta^2t<1,$
		we have that for any $N\geq 1,$
			\begin{align*}
		\e \la R(\sigma,\rho)^2\ra_{t,0}&\leq \frac{2}{N(1-2\beta^2t)}\log \frac{2}{1-2\beta^2 t}.
		\end{align*}
	\end{lem}

\begin{proof}
	Note that for $t\in (0,1)$ and any $\sigma, \sigma', \rho, \rho'\in \Sigma_N$,
	\begin{align*}
	\e \Bigl(\frac{H_N(\sigma)}{\sqrt{t}}-\frac{H_N'(\sigma)}{\sqrt{1-t}}\Bigr)H_{N,t}^1(\sigma')&=0,\\
	\e \Bigl(\frac{H_N(\rho)}{\sqrt{t}}-\frac{H_N''(\rho)}{\sqrt{1-t}}\Bigr)H_{N,t}^2(\rho')&=0,\\
	\e \Bigl(\frac{H_N(\sigma)}{\sqrt{t}}-\frac{H_N'(\sigma)}{\sqrt{1-t}}\Bigr)H_{N,t}^2(\rho')&=NR(\sigma,\rho')^2,\\
	\e \Bigl(\frac{H_N(\rho)}{\sqrt{t}}-\frac{H_N'(\rho)}{\sqrt{1-t}}\Bigr)H_{N,t}^1(\sigma')&=NR(\sigma',\rho)^2.
	\end{align*}
	Using Gaussian integration by parts gives
	\begin{align*}
	\partial_t\phi_N(t,\lambda)&=\frac{1}{2N}\beta\e \Bigl\la\Bigl(\frac{H_N(\sigma)}{\sqrt{t}}-\frac{H_N'(\sigma)}{\sqrt{1-t}}\Bigr)+\Bigl(\frac{H_N(\rho)}{\sqrt{t}}-\frac{H_N''(\rho)}{\sqrt{1-t}}\Bigr)\Bigr\ra_{t,\lambda}\\
	&=\frac{\beta^2}{2}\bigl(2\e \la R(\sigma,\rho)^2\ra_{t,\lambda}-\e\la R(\sigma^1,\rho^2)^2\ra_{t,\lambda}-\e\la R(\sigma^2,\rho^1)^2\ra_{t,\lambda}\bigr)\\
	&=\beta^2\bigl(\e \la R(\sigma,\rho)^2\ra_{t,\lambda}-\e \la R(\sigma^1,\rho^2)^2\ra_{t,\lambda}\bigr),
	\end{align*}
	where the pairs $(\sigma,\rho),$ $(\sigma^1,\rho^1),$ and $(\sigma^2,\rho^2)$ are i.i.d. copies from the Gibbs measure $\la\cdot\ra_{\lambda,t}$ and the second equality used symmetry between the distributions corresponding to the pairs $(\sigma^1,\rho^2)$ and $(\sigma^2,\rho^1).$
	Set
	\begin{align*}
	\Phi_N(t,\lambda)&=\phi_N(t,\lambda-t).
	\end{align*}
	Then from the above equation,
	\begin{align*}
	\partial_t\Phi_N(t,\lambda)&=\partial_t\phi_N(t,\lambda-t)-\partial_\lambda\phi_N(t,\lambda-t)\\
	&=\beta^2\bigl(\e \la R(\sigma,\rho)^2\ra_{t,\lambda-t}-\e \la R(\sigma^1,\rho^2)^2\ra_{t,\lambda-t}\bigr)-\beta^2\e \la R(\sigma,\rho)^2\ra_{t,\lambda-t}\\
	&=-\beta^2\e \la R(\sigma^1,\rho^2)^2\ra_{t,\lambda-t}.
	\end{align*}
	Note that $\partial_t\Phi_N(s,\lambda+t)=-\beta^2\e \la R(\sigma^1,\rho^2)^2\ra_{s,\lambda+t-s}.$
	It follows that
	\begin{align*}
	\phi_N(t,\lambda)&=\Phi_N(t,\lambda+t)\\
	&=\int_0^t\partial_{t}\Phi_N(s,\lambda+t)ds+\Phi_N(0,\lambda+t)\\
	&\leq \Phi_N(0,\lambda+t)\\
	&=\phi_N(0,\lambda+t).
	\end{align*}
	Now using the convexity of $\phi_N$ in $\lambda$ gives
	\begin{align}
	\begin{split}\label{eq1}
	\beta^2\lambda\e \la R(\sigma,\rho)^2\ra_{t,0}&=\lambda \partial_\lambda\phi_N(t,0)\\
	&\leq \phi_N(t,\lambda)-\phi_N(t,0)\\
	&\leq\phi_N(0,\lambda+t)-\phi_N(t,0)\\
	&=\phi_N(0,\lambda+t)-\phi_N(0,0),
	\end{split}
	\end{align}
	where the last equation used that $\phi_N(t,0)=\phi_N(0,0).$ Note that under the measure $\e \la \cdot\ra_{0,0}$, $\sigma$ and $\rho$ are independent uniform random variables on $\Sigma_N$ and hence, $NR(\sigma,\rho)$ is equal to the sum of $N$ i.i.d. Rademacher random variables $X_1,\ldots,X_N$ in distribution. It is well-known (see, e.g., \cite[Eq.~(A.19)]{TalagrandBook}) that 
	\begin{align*}
	\e \exp \Bigl[x\Bigl(\frac{X_1+\cdots+X_N}{\sqrt{N}}\Bigr)^2\Bigr]&\leq \frac{1}{\sqrt{1-2x}},\,\,\forall x\in [0,1/2).
	\end{align*}
	Consequently, using \eqref{eq1} and Jensen's inequality, we have
	\begin{align*}
		\beta^2\lambda\e \la R(\sigma,\rho)^2\ra_{t,0}&\leq\phi_N(0,\lambda+t)-\phi_N(0,0)\\
		&=\frac{1}{N}\e\log \Bigl\la \exp \bigl(\beta^2(\lambda+t)NR(\sigma,\rho)^2\bigr)\Bigr\ra_{0,0}\\
	&\leq \frac{1}{N}\log \e\Bigl\la \exp \bigl(\beta^2(\lambda+t)NR(\sigma,\rho)^2\bigr)\Bigr\ra_{0,0}\\
	&\leq \frac{1}{N}\log \frac{1}{\sqrt{1-2\beta^2(\lambda+t)}}
	\end{align*}
	whenever $2\beta^2(\lambda+t)<1.$ In particular, plugging $$\lambda=\frac{1}{2}\Bigl(\frac{1}{2\beta^2}-t\Bigr)$$
	into the above inequality completes our proof.
\end{proof}

\begin{proof}[Proof of Theorem~\ref{thm: sk}] Recall from \cite{Chatterjee1} that the variance of the free energy can be expressed as
	\begin{align}\label{eq2}
	\mbox{Var}(F_N(\beta))=\beta^2N\int_0^1\e \la R(\sigma,\rho)^2\ra_{t,0}dt.
	\end{align}
	For any $0<\delta<1/(2\beta^2)\leq 1,$ 
	\begin{align*}
	\int_0^{\frac{1}{2\beta^2}-\delta}\e \la R(\sigma,\rho)^2\ra_{t,0}dt&\leq \frac{1}{N}\int_0^{\frac{1}{2\beta^2}-\delta}\frac{2}{1-2\beta^2t}\log \frac{2}{1-2\beta^2 t}dt\\
	&=\frac{1}{2N\beta^2}\Bigl(\log \frac{1-2\beta^2 t}{2}\Bigr)^2\Big|_{0}^{\frac{1}{2\beta^2}-\delta}\\
	&=\frac{1}{2N\beta^2}\Bigl(\bigl(\log(\beta^2\delta)\bigr)^2-\bigl(\log 2\bigr)^2\Bigr)\\
		&\leq\frac{1}{2N\beta^2}\bigl(\log(\beta^2\delta)\bigr)^2\\
	&\leq \frac{1}{N\beta^2}\bigl( (\log \delta)^2+4(\log \beta)^2\bigr),
	\end{align*}
	where we have used the inequality $(\log(ab))^2\leq 2(\log a)^2+2(\log b)^2$ for any $a,b>0.$
	On the other hand, noting that $|R(\sigma,\rho)|\leq 1$ implies
	\begin{align*}
	\int_{\frac{1}{2\beta^2}-\delta}^1\e \la R(\sigma,\rho)^2\ra_{t,0}dt&\leq 1-\frac{1}{2\beta^2}+\delta.
	\end{align*}
	From this,
	\begin{align*}
	\mbox{Var}(F_N(\beta))&\leq \beta^2\Bigl(\frac{1}{\beta^2}\bigl( (\log \delta)^2+4(\log \beta)^2\bigr)+\Bigl(1- \frac{1}{2\beta^2}+\delta\Bigr)N\Bigr).
	\end{align*}
	If $\beta=\beta_c,$ we take $\delta=1/N$ so that 
		\begin{align*}
	\mbox{Var}(F_N(\beta))&\leq \bigl( (\log N)^2+4(\log 2)^2\bigr)+\frac{1}{2}
	\end{align*}
	and this gives the first assertion.	If $\beta^2=\beta_c^2+dN^{-\alpha}$ for $d>0$, we take $\delta=dN^{-\alpha}$ and note that
	$$
	1- \frac{1}{2\beta^2}+\delta=\frac{2dN^{-\alpha}}{1+2dN^{-\alpha}}+dN^{-\alpha}\leq 3dN^{-\alpha},
	$$
	which implies that as long as $N$ is large enough,
	\begin{align*}
\mbox{Var}(F_N(\beta))&\leq \beta^2\Bigl(\frac{1}{\beta^2}\bigl( (-\alpha\log N+\log d)^2+4(\log \beta)^2\bigr)+3dN^{1-\alpha}\Bigr)\\
&=(-\alpha\log N+\log d)^2+4(\log \beta)^2+3d\beta^2 N^{1-\alpha}
\end{align*}
and the second assertion follows.
\end{proof}

\begin{rem}\label{rmk1}\rm
	Consider the SK model at the critical temperature $\beta_c.$ Recall from Talagrand \cite[Chapter 11]{TalagrandBook} that there exists a constant $C>0$ such that $\e \la R(\sigma^1,\sigma^2)^2\ra\leq C/\sqrt{N}$ for all $N\geq 1.$ Also, it is known (see \cite{Chatterjee1}) that $t\in [0,1]\mapsto \e \la R(\sigma,\rho)^2\ra_{t,0}$ is a nondecreasing function with $\e \la R(\sigma,\rho)^2\ra_{1,0}=\e \la R(\sigma^1,\sigma^2)^2\ra.$ These imply that $\e \la R(\sigma,\rho)^2\ra_{t,0}\leq C/\sqrt{N}$ for all $t\in [0,1]$. Consequently, from \eqref{eq2}, $\mbox{Var}(F_N(\beta_c))\leq C\beta_c^2\sqrt{N}.$
\end{rem}

\smallskip

\noindent {\bf Acknowledgements.} Both authors thank Sourav Chatterjee for explaining Talagrand's upper bound for the variance of the free energy at the critical temperature and Erik Bates for the illuminating discussion and careful reading.
The first author's research is partially supported by NSF grants DMS-17-52184.

\end{document}